\documentclass[final,3p,mathptmx]{elsarticle}

\usepackage{color}
 \catcode`@=11
\usepackage{color}
\usepackage{amsthm,amssymb,amsfonts,mathrsfs}
\usepackage{amsmath}
\catcode`@=11 \@addtoreset{equation}{section}

\catcode`@=12

\newtheorem{theorem}{Theorem}[section]
\newtheorem{proposition}{Proposition}[section]
\newtheorem{lemma}{Lemma}[section]

\theoremstyle{definition}

\newtheorem{definition}{Definition}[section]
\newtheorem{remark}{Remark}

\newtheorem{example}{Example}[section]

\usepackage{enumitem}  
\usepackage{graphicx} 
\usepackage{float}

\def\X{{\bf x}}
\def\T{{\bf t}}
\def\M{{\bf m}}
\usepackage{hyperref}
\date{}
\begin{document}
\begin{frontmatter}
\title{Minimal time of the pointwise controllability for degenerate singular operators and related numerical results via B-splines}
\author[roma]{Salah Eddargani}
\ead{eddargani@mat.uniroma2.it}
\author[settat,granada]{Amine Sbai \corref{dmb}}
\cortext[dmb]{corresponding author.}
\ead{asbai@correo.ugr.es}
\address[settat]{Hassan First University of Settat, Faculty of Sciences and Technology, MISI Laboratory, B.P. 577, Settat 26000, Morocco}
\address[granada]{Department of Applied Mathematics, University of Granada, Campus de Fuentenueva s/n, 18071-Granada, Spain}
\address[roma]{Department of Mathematics, University of Rome Tor Vergata, 00133 Rome, Italy}
\begin{abstract}
The goal of this paper is to analyze the pointwise controllability properties of a one-dimensional degenerate/singular equation. We prove the conditions that characterize approximate and null controllability. Besides, a numerical simulation based on B-splines will be provided, in which the state $u$ and the control function $h$ are represented in terms of B-spline basis functions. The numerical results obtained match the theoretical ones.
\end{abstract}
\end{frontmatter}
\section{Introduction}
\label{intro}
We focus in this work on the pointwise controllability for a degenerate/singular parabolic equation, which is degenerate and singular at the boundary point $x=0$. That is to say, for
\begin{align}\label{WD}
\alpha \in [0,1) \text{ and } \mu \leq (1-\alpha)^2/4,
\end{align}
we consider the following control problem:
\begin{equation}\label{problem}
\left\{
  \begin{array}{ll}
\mathbf{A}_{\alpha, \mu} u =\delta_{\overline{b}} h(t), & (t,x) \in \omega_T=(0,T)\times(0,1),\\
u(t,0)=u(t,1)=0, & t \in (0,T), \\
u(0,x)=u_0(x), & x \in (0,1),
  \end{array}
\right.
\end{equation}
where $\mathbf{A}_{\alpha, \mu} u:=u_t - (x^\alpha u_x)_{x} - \frac{\mu}{x^{2-\alpha}}u$,  $T>0$ fixed, $u_0 \in L^{2}(0,1)$, $\alpha$ and $\mu$ are two real parameters and $\delta_{\overline{b}}$ the Dirac delta function supported in a given point $\overline{b}\in (0,1)$,  acted upon by a control function $h(t)$. 
In this case we are talking about pointwise control at $\overline{b}$.

The aim is to see if we can find a control force $h$, that acts on the system \eqref{problem}, in such a way that we can steer (or at least approximately steer) its solution towards zero equilibrium. In particular, the issues of approximate and null controllability will be treated. Furthermore, we will provide numerical simulation based on B-splines of different degrees.

Nowadays, splines represent a fundamental tools in various fields, among them numerical simulation, and particularly in the context of solving PDEs.  Splines are piecewise-defined functions that consists of polynomial segments joined together smoothly at specific points called knots.  They  serve as a flexible and efficient means of discretizing the solution space. Their piecewise polynomial representation of functions, allowing for local control and refinement of the approximation. This flexibility is crucial when dealing with the spatial discretization of PDEs, as it enables adaptive refinement in regions of interest while maintaining a computationally efficient representation in less critical areas.

Splines are used also in the context of controllability simulation of PDEs. Indeed, the controllability in the context of PDEs refers to the ability to manipulate the system's state through the application of external controls. Splines can play a significant role in achieving controllability by providing a framework for representing control inputs and optimizing their distribution throughout the simulation domain. The piecewise nature of splines, as well as the non-negative partition of unity of B-splines (or basis splines) allow for the design of control strategies that can be localized in both space and time, offering finer control over the system's behavior.

Now, we highlight definitions \ref{def ac} and \ref{def nc} of approximate controllability and null controllability, respectively.
\begin{definition}[Approximate controllability]\label{def ac}
Equation \eqref{problem} is said to be approximately controllable at target time $T>0$, if for any $\varepsilon > 0$ and $u_0, u_T \in L^{2}(0,1)$,
there exists a control $h \in L^2(0,T)$ where the solution $y$ to \eqref{problem} satisfies
\begin{equation}\label{ac}
\|u(T)-u_T\|_{L^{2}(0,1)} \leq \varepsilon.
\end{equation}
\end{definition}
Let us define system \eqref{adjproblem} as the backward adjoint problem of \eqref{problem}:\\
\begin{equation}\label{adjproblem}
\left\{
  \begin{array}{ll}
\varphi_t + (x^\alpha\varphi_x)_{x} + \frac{\mu}{x^{2-\alpha}}\varphi  =0, & (t,x) \in \omega_T,\\
\varphi(t,0)=\varphi(t,1)=0, & t \in (0,T), \\
\varphi(T,x)=\varphi_0(x), & x \in (0,1),
  \end{array}
\right.
\end{equation}
where $\varphi_0 \in L^2(0,1)$ is a given initial datum.
Then, it is well known in \cite[Theorem 2.43]{coron} that the last issue of approximate controllability of \eqref{problem} can be reduced to the unique continuation property
for the adjoint equation \eqref{adjproblem}.
\begin{proposition}\label{prop1}
Equation \eqref{problem} is approximately controllable at time $T>0$ if and only if its adjoint equation \eqref{adjproblem}
satisfies the property:
\begin{equation}\label{ucproperty}
\text{For all } \varphi_0 \in L^2(0,1) \text{ s.t. } \varphi(\cdot ,\overline{b})=0
\text{ on } (0,T) \Rightarrow \varphi_0=0 \quad \text{in}\quad (0,1).
\end{equation}
\end{proposition}
\begin{definition}[Null controllability]\label{def nc} 
Null controllability at a given time $T>0$ for equation \eqref{problem} holds true if, for every $u_0 \in L^{2}(0,1)$, there exists a control $h \in L^2(0,T)$ such that $u$ the solution of \eqref{problem} fulfills
\begin{equation}\label{ncp}
u(T,x)=0 \, \text{ for all } x \in (0,1).
\end{equation}
\end{definition}
\subsection{Background and motivations}
Before going any further, let us frame the operator $\mathbf{A}_{\alpha, \mu}$, which is purely degenerate when $\mu=0$, but when $\alpha=0$, this operator becomes purely singular as shown below
\begin{align}
\mathbf{A}_\alpha u:=u_t-\left(x^\alpha u_x\right)_x, \quad x \in(0,1), \tag{PD} \label{PD} \\
\mathbf{A}_\mu u:=u_t-u_{x x}-\frac{\mu}{x^2} u, \quad x \in(0,1). \tag{PS} \label{PS}
\end{align}
Null controllability of the operator $\mathbf{A}_{\alpha, \mu}$ was investigated in \cite{Van2011}. In particular, the authors proved that the problem is null controllable iff:
$\alpha \in [0,2)$ and $\mu \leq \mu(\alpha)$,
where $\mu(\alpha):=\frac{(1-\alpha)^2}{4}$ appears in the following Hardy inequality
\begin{align}\label{hardy}
\frac{(1-\alpha)^2}{4} \int_0^1 \frac{z^2}{x^{2-\alpha}} d x \leq \int_0^1 x^\alpha z_x^2 d x .
\end{align}
We refer to \cite{AHSS2022, gms2024, gms2024stab} for other situations on this theme.

Now, let us turn to the pointwise control problems. We recall that S. Dolecki has discussed the minimal control time  in \cite{dol} for the null controllability chalenge of parabolic equations. In particular, the control was supported in a given point $\overline{b}$ in the interior of the domain as in our system \eqref{problem} where particularly $\alpha=\mu=0$.
Using the well-known moment method \cite{fatrus1971}, S. Dolecki has determined a minimal time $T_0 \in[0,+\infty]$ such that
\begin{equation*}\label{T0}
T_0(\overline{b})=\limsup _{n \rightarrow \infty}-\frac{\log (|\sin (n \overline{b})|)}{n^2}.
\end{equation*}
Obviously, we can remark that $T_0(\overline{b})$ depends on the position of the control, since null controllability is guaranteed for any target time $T>T_0$, but if $T<T_0$, the problem \eqref{problem} where $\alpha=\mu=0$, cannot be null controllable. 

In the same direction, in \cite{allal2020}, the pointwise null controllability for \eqref{problem} with $\mu = 0$ was considered. The authors have described the minimal time for the purely degenerate operator $\mathbf{A}_\alpha$ (i.e. the \eqref{PD} case), where it depends on $\overline{b}$ (control position) and the degeneracy parameter. In particular
\begin{equation*}\label{T0alpha}
T^{(\alpha)}_0(\overline{b})=\limsup _{n \rightarrow+\infty}-\frac{\log \left(\left|\Phi_{v_\alpha, n}\left(\overline{b}\right)\right|\right)}{\lambda_{v_\alpha, n}},
\end{equation*} 
where $\left(\lambda_{v_\alpha, k},\Phi_{v_\alpha, k}\right)_{k \geq 1}$ represent the spectrum of $\mathbf{A}_{\alpha}$. More precisely, the authors proved that the null controllability holds when $T>T^{(\alpha)}_0$, and fails when $T<T^{(\alpha)}_0$.

The goal of this paper is to investigate the same kind of issues, but for degenerate singular equations and in particular to deal with both approximate and null controllability characterization for our pointwise problem \eqref{problem} with respect to the degeneracy and singularity parameters. 

The problem \eqref{problem} considered in this manuscript, has never been treated before.

\subsection{Main results}
Two controllability results for the equation \eqref{problem} will be treated in this work. We show first the geometric condition \eqref{cond1}, that gives a guarantee of approximate controllability at $T>0$ for the equation \eqref{problem}. As a second step, considering the same condition, we show $T^{(\alpha, \mu)}_0(\overline{b}) \in [0,+\infty]$ (see \eqref{mintime}) such that: 
\begin{center}
{\sl Equation \eqref{problem} is null controllable if $T >T^{(\alpha, \mu)}_0(\overline{b})$ and is not for $T <T^{(\alpha, \mu)}_0(\overline{b})$.}
\end{center}
To do so, Assuming \eqref{WD} and let $\nu(\alpha, \mu)=\frac{2}{2-\alpha}\sqrt{\mu(\alpha)-\mu}.$ Thus, we describe the set $\mathcal{P}$ as
\begin{equation}\label{set}
\mathcal{P}= \left\{ \left( \dfrac{j_{\nu(\alpha, \mu), k} }{j_{\nu(\alpha, \mu), n}} \right)^{\frac{2}{2-\alpha}},\quad n > k \geq 1\right\},
\end{equation}
where $(j_{\nu(\alpha, \mu),k})_{k\geq 1}$ represents the Bessel functions zeros (see Section \ref{CP}). Using this notations, we can give our result of approximate controllability.
\begin{theorem}\label{acresult}
Assuming \eqref{WD}. Approximate controllability at $T > 0$ for equation \eqref{problem} holds true iff:
\begin{equation}\label{cond1}
\overline{b} \notin \mathcal{P}.\tag{H1}
\end{equation}
\end{theorem}
Moreover, the following result holds.
\begin{theorem}\label{ncresult}
Assuming \eqref{WD} and let $u_0 \in L^2(0,1)$. Assume that \eqref{cond1} holds and take
\begin{equation}\label{mintime}
T^{(\alpha , \mu)}(\overline{b}) := \limsup\limits_{k \to + \infty} -  \frac{\log(|\Phi_{\alpha,\mu, k}(\overline{b})| )}{\lambda_{\alpha,\mu, k}},
\end{equation}
where $(\Phi_{\alpha,\mu, k})_{k\geq 1}$ and $(\lambda_{\alpha,\mu, k})_{k\geq 1}$ are respectively the sequence of eigenvectors and its associated eigenvalues for the specral problem related to \eqref{problem} (see section \ref{CP}).
Then, for $T > 0$, one can obtain:
\begin{enumerate}[label=(\roman*)]
\item If $T< T^{(\alpha , \mu)}(\overline{b})$, equation \eqref{problem} cannot be null controllable $T$.
\item If $T> T^{(\alpha , \mu)}(\overline{b})$, null controllability at $T$ of equation \eqref{problem} holds true.
\end{enumerate}
\end{theorem}
\begin{remark}\label{remarkobs}
By duality, the equation \eqref{problem} is null controllable at time $T$ is equivalent to the observability estimate:
\begin{equation}\label{oi1}
\|\varphi(0, \cdot)\|_{L^2(0,1)}^2 \leq C \int_0^T \varphi(t, \overline{b})^2 d t,
\end{equation}
where  $\varphi$ is a solution of problem \eqref{adjproblem} and $C>0$ (constant). 
\end{remark}
To demonstrate Theorem \ref{ncresult}, we need to reduce null controllability issue into a moment problem (see Section \ref{PNC}). To do so, we use the following result (see \cite{ABGT2011, FBT}). 
\begin{theorem}\label{thm-biorth}
Given $T>0$. Assume that $ \{\Lambda_k\}_{k\geq1} \subset \mathbb{R}_{+} $ fulfills 
\begin{equation}\label{separability}
\sum\limits_{k\geq 1} \dfrac{1}{|\Lambda_k|} < + \infty \text{ and }
|\Lambda_k - \Lambda_l| \geq \rho |k - l|, \quad \forall k, l \geq 1.
\end{equation}
for a constant $\rho >0$.
Then, there exists a family $\{q_k \}_{k\geq 1} \subset L^2(0, T) $ biorthogonal to $\{e^{- \Lambda_k t}\}_{k\geq 1}$.
Furthermore, we have
\begin{equation}\label{bound}
\forall \varepsilon > 0,\, \exists C_{\varepsilon} > 0,\, \|q_k \|_{L^2(0, T)} \leq C_\varepsilon e^{\varepsilon \Lambda_k}, \quad \forall k \geq 1.
\end{equation}
\end{theorem}

\section{Well-posedness and spectral properties}
\label{CP}
In this section we will give some existence and uniqueness results of solutions for the degenerate singular equation \eqref{problem} as well as some useful spectral properties. 
\subsection{Functional framework}
Let us assume \eqref{WD} and consider the Hilbert space $H_{\alpha}^{1, \mu}$ with its scalar product
\begin{align*}
\begin{cases}
H_{\alpha}^{1, \mu}(0,1):=\left\{z \in L^{2}(0,1) \cap H_{l o c}^{1}(0,1] : \int_{0}^{1}\left(x^{\alpha} z_{x}^{2}-\mu \frac{z^{2}}{x^{2-\alpha}}\right) d x<+\infty\right\},\\
\langle v, w\rangle_{H_\alpha^{1, \mu}}:=\int_0^1 v w+x^\alpha v_x w_x-\frac{\mu}{x^{2-\alpha}} v w d x . \end{cases}
\end{align*}
As mentioned in \cite{Van2011},the trace exists for any $u$ in $ H_\alpha^{1,\mu}(0,1)$ at $x=1$ and also for $x=0$ in the weak degenerate scenario.
This enables us to define the space
\begin{align*}
H_{\alpha, 0}^{1, \mu}(0,1):=\left\{z \in H_{\alpha}^{1, \mu}(0,1) \mid z(0)=z(1)=0\right\}.
\end{align*}
Since $C_{c}^{\infty}(0,1)$ is dense both in $L^{2}(0,1)$ and in $H_{\alpha, 0}^{1, \mu}(0,1)$, $H_{\alpha, 0}^{1, \mu}(0,1)$ is dense in $L^{2}(0,1)$.

Further, we define $H_{\alpha, 0}^{-1, \mu}(0,1)$ the dual space of $H_{\alpha, 0}^{1, \mu}(0,1)$ with respect to the pivot space $L^{2}(0,1)$, with the following norm
\begin{align*}
\|f\|_{H_{\alpha, 0}^{-1, \mu}(0,1)}:=\sup _{\|g\|_{H_{\alpha, 0}^{1, \mu}(0,1)}=1}\langle f, g\rangle_{H_{\alpha, 0}^{-1, \mu}(0,1), H_{\alpha, 0}^{1, \mu}(0,1)}.
\end{align*}
Thanks to the generalized Hardy inequality \eqref{hardy}, $H_{\alpha, 0}^{1, \mu}(0,1)=H_{\alpha, 0}^1(0,1)$ for the subcritical parameter case: $\mu<\mu(\alpha)$. But, for the critical case, one has (see \cite{VazZua2000} for $\alpha=0$ ):
\begin{align*}
H_{\alpha, 0}^1(0,1) \underset{\neq}{\subset} H_{\alpha, 0}^{1, \mu(\alpha)}(0,1) .
\end{align*}
\subsection{Well-posedness}
Let us introduce the following unbounded operator
\begin{align*}
\mathbf{A}: D(\mathbf{A}) \subset L^2(0,1) &\rightarrow L^2(0,1)\\
u &\mapsto \mathbf{A} u:=\left(x^\alpha u_x\right)_x+\frac{\mu}{x^{2-\alpha}}u,
\end{align*}
with its domain
\begin{align*}
D\left(\mathbf{A}\right):=\left\{u \in H_{\alpha, 0}^{1, \mu}(0,1) \cap H_{\mathrm{loc}}^2((0,1]): \left(x^\alpha u_x\right)_x+\frac{\mu}{x^{2-\alpha}} u \in L^2(0,1)\right\}.
\end{align*}
If $u \in D\left(\mathbf{A}\right)$, then the boundary conditions $u(0)=u(1)=0$ are satisfied. Moreover, the bilinear form associated to $\mathbf{A}$ is coercive in $H_{\alpha, 0}^{1,\mu}(0,1)$ (see \cite{Van2011} for more details about $\mu=\mu(\alpha)$ and $\mu < \mu(\alpha)$). This allows to say that problem \eqref{problem} is well-posed. In particular we have the following theorem.
\begin{theorem}{\cite[Theorem III.1.2]{lions1971}}
For any initial data $u_0$ in $L^2(0,1)$ and $h$ in $L^2(0, T)$, equation \eqref{problem} has a unique solution $u$, such that $u \in L^2\left(0, T ; H_\alpha^1(0,1)\right) \cap C^0\left([0, T] ; L^2(0,1)\right)$ and
$$
\begin{aligned}
\|u\|_{L^2\left(0, T ; H_{\alpha,0}^{1,\mu}(0,1)\right)}+\|u\|_{C^0\left([0, T] ; L^2(0,1)\right)}  \leq C\left(\left\|u_0\right\|_{L^2(0,1)}+\left\|\delta_{\overline{b}}\right\|_{H_{\alpha,0}^{-1,\mu}}\|h\|_{L^2(0, T)}\right),
\end{aligned}
$$
for a constant $C>0$.
\end{theorem}
\subsection{Spectral analysis of $\mathbf{A}$}\label{EE}
This part will be dedicated to treat the eigenvalue problem related to the operator $u \mapsto  - (x^\alpha u_x)_x - \frac{\mu}{x^{2-\alpha}}u$, which will be useful for the rest of this paper. 
	
As we can see in the next proposition \ref{spectrum}, the spectrum will be described in terms of Bessel functions $J_\nu$ (with order $\nu \in \mathbb{R}_{+}$ \cite{Watson}), which satisfy the following (ODE)
\begin{align*}
x^2 g''(x) + x g'(x) + (x^2 - \nu^2) g(x) = 0, \qquad  x\in (0, + \infty),
\end{align*}
and its zeros $(j_{\nu,n})_{n\geq 1}$ (see Figure \ref{fig:bessel}) that satisfy the following lower-upper bounds (\cite{LM2008}):
\begin{align}	
\forall \nu \in \Big[0, \dfrac{1}{2}\Big],\, \forall n \geq 1,\quad	\big( n + \frac{\nu}{2} - \frac{1}{4} \big) \pi \leq j_{\nu, n } \leq \big( n + \frac{\nu}{4} - \frac{1}{8} \big) \pi, \label{boundcase1}\\
\forall \nu \geq \dfrac{1}{2},\, \forall n \geq 1, \quad \big( n + \frac{\nu}{4} - \frac{1}{8} \big) \pi \leq j_{\nu, n } \leq \big( n + \frac{\nu}{2} - \frac{1}{4} \big) \pi.\label{boundcase2}
\end{align}
\begin{figure}[H]
\centering
\includegraphics[scale=0.5]{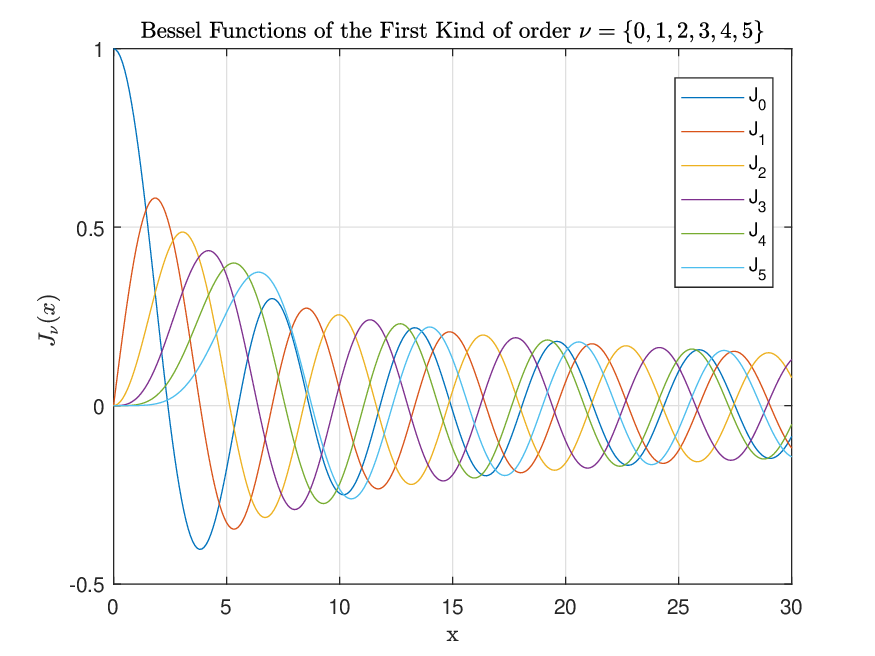}
\caption{Bessel functions of first kind.}
\centering 
\label{fig:bessel}
\end{figure}
As in \cite{biccsantavan2020}, we give the description of the eigenvalues and its associated eigenfunctions.
\begin{proposition}\label{spectrum} For any $\mu\leq  \mu(\alpha)$, the admissible spectrum $(\lambda, \Phi)$ related to the operator $u \mapsto  - (x^\alpha u_x)_x - \frac{\mu}{x^{2-\alpha}}u$, are the following:
\begin{equation}\label{eigenvalues}
\lambda_{\alpha, \mu, n} = \Big( \dfrac{2-\alpha}{2}  \Big)^2 (j_{\nu(\alpha,\mu),n})^2 \qquad  \forall n \geq 1,
\end{equation}
\begin{equation}\label{eigenfunctions}
\Phi_{\alpha,\mu, n}(x) =  \frac{\sqrt{2-\alpha}}{|J_{\nu(\alpha,\mu)}'(j_{\nu(\alpha,\mu), n})|}
x^{\frac{1-\alpha}{2}} J_{\nu(\alpha,\mu)}\Big(j_{\nu(\alpha,\mu), n}x^{\frac{2-\alpha}{2}}\Big), \quad n \geq 1 .
\end{equation}
Furthermore, $(\Phi_{\alpha,\mu, n})_{n\geq 1} $ constitutes an orthonormal basis of $L^2(0,1)$.
\end{proposition}
Before ending this section, we present Lemma \ref{gapresult}, which will be proved in \ref{appendix:b}.
\begin{lemma}\label{gapresult}  
The sequence $(\lambda_{\alpha,\mu, k})_{k\geq 1}$ given by \eqref{eigenvalues}, satisfies:
\begin{itemize}
\item $\displaystyle\sum_{n\geq 1}  \frac{1}{\lambda_{\alpha,\mu,n}}$ is convergent.
\item  $\forall n,m \in \mathbb{N}^\star$, $\exists \rho > 0$ (constant):
\begin{align}\label{gap}
 | \lambda_{\alpha,\mu, n} - \lambda_{\alpha,\mu, m} | \geq \rho |n^2 - m^2|. \tag{gap}
\end{align}
\end{itemize}
\end{lemma}
\section{Pointwise approximate controllability}
\label{PAC}
The scenario of this section will be around Theorem \ref{acresult} and its proof.
Due to Proposition \ref{prop1}, approximate controllability is equivalent to the property \eqref{ucproperty} for the adjoint equation \eqref{adjproblem}.

Before going any further, we first state the following tool result:
\begin{theorem}{\cite[Lemma 3.1, Existence of biorthogonal family]{FBT}}\label{biorth}
Let $(\lambda_{\alpha,\mu,k})_{k\geq 1}$ as in \eqref{eigenvalues} and $T\in \big(0,+\infty\big]$ fixed.
Then, there exists $\{q_{\alpha,\mu, k}\}_{k\geq 1}$ in $L^2(0,T)$ biorthogonal to $\{e^{-\lambda_{\alpha,\mu,k} t}\}_{k\geq 1}$.
Moreover, we have
\begin{equation}\label{l2bound}
 \forall \varepsilon >0, \exists C_{\varepsilon,T}>0: \, \|q_{\alpha,\mu,k}\|_{L^2(0,T)}
\leq C_{\varepsilon,T} e^{\varepsilon \lambda_{\alpha,\mu,k}}, \quad \forall k\geq 1.
\end{equation}
\end{theorem}
\begin{proof}
\noindent\textbf{Sufficient condition:} Suppose \eqref{cond1} is valid and prove the property \eqref{ucproperty}. To do so, considering $\varphi_0 \in L^2(0,1)$ and suppose $\varphi$ its corresponding solution of \eqref{adjproblem} fulfils:
\begin{align}\label{Hyp ac}
\varphi(t,\overline{b})=0, \qquad \text{for all }t\in (0,T).
\end{align}
We can expand the solution $\varphi$ of \eqref{adjproblem} in $(\Phi_{\alpha,\mu, k})_{k\geq 1}$, one has
\begin{align*}
\varphi(t,x)=\sum_{k\geq 1} \varphi^0_{\alpha,\mu,k} e^{-\lambda_{\alpha,\mu,k}(T-t)} \Phi_{\alpha,\mu, k}(x).
\end{align*}
As consequence of \eqref{Hyp ac} and Theorem \ref{biorth}, the following identity holds true:
\begin{align*}
\varphi^0_{\alpha,\mu,k} \Phi_{\alpha,\mu, k}(\overline{b})=0,\quad \forall k\geq1,
\end{align*}

thus $\varphi_0=0$. Then equation \eqref{problem} is approximately controllable  by virtue of Proposition \ref{prop1}.

\smallskip
\noindent\textbf{Necessary condition:}
We assume that hypothesis \eqref{cond1} doesn't hold. One has
\begin{align*}
\exists n_0> k_0 \geq 1,\text{ for which: } \overline{b}=\left(\frac{j_{\nu(\alpha,\mu), k_0}}{j_{\nu(\alpha,\mu), n_0}}\right)^{\frac{2}{2-\alpha}}.
\end{align*}
The solution $\varphi$ of equation \eqref{adjproblem} associated to any initial data $\varphi_0$ can be explicitly given by
\begin{align*}
\varphi(t,x)=\sum_{k\geq 1} \varphi^0_{\alpha,\mu,k} e^{-\lambda_{\alpha,\mu,k}(T-t)} \Phi_{\alpha,\mu, k}(x).
\end{align*}
Then for $\varphi_0= \Phi_{\alpha,\mu, n_0} \in L^2(0,1)$, the solution of \eqref{adjproblem} related to this choice satisfy
\begin{align*}
\varphi_{\alpha,\mu,n_0}(t,x) =  e^{-\lambda_{\alpha,\mu, n_0}(T-t)} \Phi_{\alpha,\mu, n_0}(x),\qquad (t,x)\in \omega_T.
\end{align*}
this contradicts the assumption, since $\varphi_0\neq 0$. This ends the proof.
\end{proof}

\section{Pointwise null controllability}
\label{PNC}
The negative result of null controllability will be presented in the first part \ref{negativecontrol}. On the other hand, the second part \ref{positivecontrol} deals with the positive result, where the solution of the equation \eqref{problem} can be steered to zero with respect to the minimal time $T^{(\alpha ,\mu)}(\overline{b})$. Then the proof of Theorem \ref{ncresult} will be completed. To this aim, let us assume that $\overline{b}$ fulfills \eqref{cond1}.
\subsection{Negative result}\label{negativecontrol}
Assuming that $T\in (0,T^{(\alpha, \mu)}_0(\overline{b}))$, where $T^{(\alpha, \mu)}_0(\overline{b})$ is given by \eqref{mintime} and recall that condition \eqref{cond1} holds. Using some ideas from \cite{Khodja2016} and arguing by contradiction,
let us suppose that equation \eqref{problem} is null controllable at $T<T^{(\alpha, \mu)}_0(\overline{b})$, then the observability inequality \eqref{oi1} is satisfied for any $\varphi$ solution of adjoint equation \eqref{adjproblem} with respect to a positive constant $C$ (see Remark \ref{remarkobs}).

Considering $\varphi_{\alpha,\mu,k}(t,x) = e^{-\lambda_{\alpha,\mu, k}(T- t)}  \Phi_{\alpha,\mu, k}(x)$, for all $k\geq 1$, the solution of \eqref{adjproblem} related to $\varphi_0 = \Phi_{\alpha,\mu, k}$.
Then, \eqref{oi1} becomes
\begin{align*}
e^{-2 \lambda_{\alpha,\mu, k} T} &\leq C  \Phi_{\alpha,\mu, k}(\overline{b})^2  \int_0^T e^{-2 \lambda_{\alpha,\mu, k} (T-t)}\, dt \nonumber\\
&\leq C \frac{1}{2\lambda_{\alpha,\mu, k}}\big(1- e^{-2 \lambda_{\alpha,\mu, k}T}\big)\Phi_{\alpha,\mu, k}(\overline{b})^2 \nonumber\\
& \leq \frac{C}{2\lambda_{\alpha,\mu, 1}} \Phi_{\alpha,\mu, k}(\overline{b})^2,\quad \forall k\geq 1,
\end{align*}
which implies 
\begin{align}\label{oi2}
1 \leq  C' e^{2 \lambda_{\alpha,\mu, k} T}  \Phi_{\alpha,\mu, k}(\overline{b})^2\qquad \forall k\geq 1,
\end{align}
where $C'$ does not depend on $k$. Moreover, from the definition of $T^{(\alpha, \mu)}_0(\overline{b})$, one has
\begin{align*}
\exists \{k_n\}_{n\geq1} \, \text{of positive integers:}\quad T^{(\alpha, \mu)}_0(\overline{b})= \lim_{n \to + \infty} -  \frac{\log(|\Phi_{\alpha,\mu, k_n}(\overline{b})| )}{\lambda_{\alpha,\mu, k_n}}.
\end{align*}
If $0<T^{(\alpha, \mu)}_0(\overline{b})< +\infty$, then we can obtain
\begin{align*}
\forall \varepsilon>0, \exists n_{\varepsilon}\geq 1, \, \text{such that:}\quad
T^{(\alpha, \mu)}_0(\overline{b}) - \varepsilon \leq -\frac{\log(\big| \Phi_{\alpha,\mu, k_n}(\overline{b})\big| )}{\lambda_{\alpha,\mu, k_n}}, \qquad \forall n \geq n_{\varepsilon}.
\end{align*}
Thus \eqref{oi2} implies
\begin{align*}
1 \leq  C' e^{- 2\lambda_{\alpha,\mu, k_n} (T^{(\alpha, \mu)}_0(\overline{b}) - T -\varepsilon)}, \qquad \forall n \geq n_{\varepsilon}.
\end{align*}
Then we obtain a contradiction with $\varepsilon\in (0, T-T^{(\alpha, \mu)}_0(\overline{b}))$. This confirms the negative result of Theorem \ref{ncresult}.

\subsection{Positive result}\label{positivecontrol}
In this part, we shall find a force control $h\in L^{2}(0,T)$, such that the null controllability condition \eqref{ncp} holds true at $T> 0$ ($T> T^{(\alpha, \mu)}_0(\overline{b})$) for any solution of \eqref{problem} with respect to the initial data $u_0 \in L^{2}(0,1)$. To achieve this, we transform the null controllability result of equation \eqref{problem} into the moment problem \eqref{momentpb0} (see \cite{fatrus1971} for more details).

Firstly, the solution $u$ of \eqref{problem} associated with $u_0(x) = \sum\limits_{k\geq 1}u^0_{\alpha,\mu,k}\Phi_{\alpha,\mu, k}(x)$, is as follows
\begin{equation*}
u(t,x) = \sum_{k\geq 1}u_{\alpha,\mu,k}(t)\Phi_{\alpha,\mu, k}(x),\quad \text{where}\quad
\sum_{k\geq 1}u_{\alpha,\mu,k}^2(t)<+\infty.
\end{equation*}
Characterization \eqref{ncp} of null controllability, can be reformulated as:
\begin{equation}\label{ncp1}
u_{\alpha,\mu,k}(T)=0,\quad \text{for all }k\geq 1.
\end{equation}
The following adjoint problem:
\begin{equation}\label{adjpb}
\left\{
\begin{array}{lll}
(\varphi_{\alpha,\mu,k})_t + (x^\alpha (\varphi_{\alpha,\mu,k})_{x})_x +\frac{\mu}{x^{2-\alpha}}\varphi_{\alpha,\mu,k} =0 , & & \text{ in } Q,\\
\varphi_{\alpha,\mu,k}(t, 0)=\varphi_{\alpha,\mu,k}(t,1)= 0, & & \text{ on } (0,T),\\
\varphi_{\alpha,\mu,k}(T, x)=\Phi_{\alpha,\mu, k}(x), & & \text{ in } (0,1),
\end{array}
\right.
\end{equation}
has the solution $\varphi_{\alpha,\mu,k}(t,x):= e^{-\lambda_{\alpha,\mu,k}(T-t)} \Phi_{\alpha,\mu, k}(x)$. Then from \eqref{problem} and \eqref{adjpb}, one has
\begin{footnotesize}
\begin{align*}
\Phi_{\alpha,\mu, k}(\overline{b}) \int_0^T  h(t) e^{-\lambda_{\alpha,\mu, k}(T- t)}\,dt
&\quad= \int_{0}^{1}\varphi_{\alpha,\mu,k}u|_{0}^{T}\,dx + \int_{0}^{T}(\varphi_{\alpha,\mu,k})_x u_x|_{0}^{1}\,dt - \int_{0}^{T}u_x (\varphi_{\alpha,\mu,k})_x|_{0}^{1}\,dt\\
&\quad= \int_{0}^{1} u(T,x)\Phi_{\alpha,\mu, k}(x)\,dx - \int_{0}^{1} u(0,x)\Phi_{\alpha,\mu, k}(x)e^{-\lambda_{\alpha,\mu,k}T}\,dx  \\
&\quad= u_{\alpha,\mu,k}(T) - u^{0}_{\alpha,\mu,k} e^{-\lambda_{\alpha,\mu,k}T}.
\end{align*}
\end{footnotesize}
Null controllability issue for equation \eqref{problem} can be transformed to the problem given by:
\begin{align}\label{momentpb0}
\text{Find } h\in L^2(0,T):\, \Phi_{\alpha,\mu, k}(\overline{b})\int_0^T  h(t) e^{-\lambda_{\alpha,\mu, k}(T- t)}\,dt=- u^{0}_{\alpha,\mu,k} e^{-\lambda_{\alpha,\mu, k} T}, \, \forall k \geq 1.\tag{MP}
\end{align}
The following condition fulfilled by \eqref{cond1}, is necessary to solve \eqref{momentpb0} 
\begin{align*}
\Phi_{\alpha,\mu, k}(\overline{b}) \neq 0,\quad \forall k\geq 1.
\end{align*}
Considering $v(t):=h(T-t)\in L^2(0,T)$, \eqref{momentpb0} reads as follows:
\begin{equation}\label{momentpb1}
\text{Find } h\in L^2(0,T):\, \int_0^T v(t) e^{-\lambda_{\alpha,\mu,k}t}\,dt
= -\dfrac{e^{-\lambda_{\alpha,\mu, k} T} u^{0}_{\alpha,\mu,k}}{\Phi_{\alpha,\mu, k}(\overline{b})},\quad \forall k \geq 1.
\end{equation}
Using Theorem \ref{biorth}, then we are looking for a solution $v(t)=h(T-t)$ to \eqref{momentpb1} as
$v(t)= \sum_{k\geq 1} u_k q_{\alpha,\mu, k}(t),$
for some unknown coefficients $v_k\in \mathbb{R}$ ($k\geq 1$). Then
\begin{equation*}
v_k= -\dfrac{e^{-\lambda_{\alpha,\mu ,k} T}u^{0}_{\alpha,\mu,k}}{\Phi_{\alpha,\mu, k}(\overline{b})},
\end{equation*}
and thus, we define a formal solution $h$ to the moment problem \eqref{momentpb1} as
\begin{equation}\label{control}
h(t):=v(T-t)= - \sum_{k\geq 1}  \dfrac{e^{-\lambda_{\alpha,\mu ,k} T}u^{0}_{\alpha,\mu,k}}{\Phi_{\alpha,\mu, k}(\overline{b})} q_{\alpha,\mu, k}(T-t).
\end{equation}
We have now to prove that $h \in L^2(0,T)$. For any $\varepsilon > 0$ fixed, by definition of $T^{(\alpha, \mu)}_0(\overline{b})$ (see \eqref{mintime}), one has 
\begin{equation*}
\exists C_{\alpha,\mu, \varepsilon}>0 , \, \text{such that: } \dfrac{1}{|\Phi_{\alpha,\mu, k}(\overline{b})|}
\leq C_{\alpha,\mu, \varepsilon} e^{\lambda_{\alpha,\mu, k} ( T^{(\alpha, \mu)}_0(\overline{b}) +\varepsilon)}, \qquad \forall k \geq 1.
\end{equation*}
Using this last inequality and taking into account the bound \eqref{l2bound}, we get
\begin{align*}
\left\|\dfrac{e^{-\lambda_{\alpha,\mu ,k} T}u^{0}_{\alpha,\mu,k}}{\Phi_{\alpha,\mu, k}(\overline{b})} q_{\mu, k}\right\|_{L^2(0,T)}
\leq C_{\alpha,\mu, \varepsilon, T} \|u_0\|_{L^2(0,1)} e^{-\lambda_{\alpha,\mu, k} (T - T^{(\alpha, \mu)}_0(\overline{b}) -2\varepsilon )},
\end{align*}
where $C_{\alpha,\mu, \varepsilon, T}>0$. Choosing $\varepsilon = \frac{T- T^{(\alpha, \mu)}_0(\overline{b})}{4}$ in the estimate above, then
\begin{equation*}
\exists C_{\alpha,\mu, T}'>0, \quad \left\|\dfrac{e^{-\lambda_{\alpha,\mu ,k} T}u^{0}_{\alpha,\mu,k}}{\Phi_{\alpha,\mu, k}(\overline{b})} q_{\alpha,\mu, k}\right\|_{L^2(0,T)}  \leq C_{\alpha,\mu, T}' \|u_0\|_{L^2(0,1)} e^{-\lambda_{\alpha,\mu, k} (T - T^{(\alpha, \mu)}_0(\overline{b}))/2}.
\end{equation*}
Since we have
\begin{footnotesize}
\begin{align*}
\sum_{k\geq 1}& e^{-\lambda_{\alpha,\mu, k} (T - T^{(\alpha, \mu)}_0(\overline{b}))/2}
\quad = \frac{2}{T - T^{(\alpha, \mu)}_0(\overline{b})} \sum_{k\geq 1}\Big(\lambda_{\alpha,\mu, k} \frac{T - T^{(\alpha, \mu)}_0(\overline{b})}{2}
e^{-\lambda_{\alpha,\mu, k} (T - T^{(\alpha, \mu)}_0(\overline{b}))/2}\Big) \frac{1}{\lambda_{\alpha,\mu, k}} < +\infty.\end{align*}
\end{footnotesize}
One has
\begin{equation*}
\sum_{k\geq 1} \left\|\dfrac{e^{-\lambda_{\alpha,\mu ,k} T}u^{0}_{\alpha,\mu,k}}{\Phi_{\alpha,\mu, k}(\overline{b})} q_{\alpha,\mu, k}\right\|_{L^2(0,T)}
\leq C_{\alpha,\mu, T} \|u_0\|_{L^2(0,1)} \sum_{k\geq 1} e^{-\lambda_{\alpha,\mu, k} (T - T^{(\alpha, \mu)}_0(\overline{b}))/2} < + \infty.
\end{equation*}
Theorem \ref{ncresult} is therefore proved, since the last inequality ensures that $h\in L^2(0,T).$

\section{ Finite element scheme for the state equation based on B-splines}
In this part, we review some results around spline spaces \cite{B01} and spline collocation framework \cite{BLL19}, providing necessary background to make the paper self-contained.

Consider the set of breakpoints $\X := \left\lbrace 0=x_0 < x_1 \ldots < x_n=1 \right\rbrace$. Let $S_{d, \X}$ denote the space of piecewise polynomial functions of degree less than or equal to $d$ associated with the partition of the interval $[0, 1]$ induced by $\X$. Let $\M := \left\lbrace m_i \right\rbrace_{i=1}^{n-1}$ be a set of values with entries $m_i \leq d+1$. We define the subspace $S_{d,\X, \M}$ as the set of functions in $S_{d, \X}$ that are $C^{m_i-1}$-smooth at $x_i$, for $i=1,\ldots, n-1$. The dimension of $S_{d,\X, \M}$ is equal to

\[
N :=n(d+1)- \sum_ {i=1}^{n-1} m_i
\]
We denote by $\T := \left\lbrace t_i \right\rbrace_{i=1}^{N+d+1}$ the breakpoint sequence formed by the breakpoints $x_i$ repeated $d+1-m_i$ times, where $t_1 \leq t_2 \leq \ldots \leq t_{d+1} \leq x_0$ and $x_n \leq t_{N+1} \leq \ldots \leq t_{N+d+1}$. Here, the entries $d+1-m_i$ and $m_i-1$ represent the knot multiplicity of $x_i$ and the spline smoothness order at $x_i$, respectively.

A normalized B-spline (e.g. non-negative partition of unity) of degree $d$ associated with $\T$, is defined as:
\[
B_{k,d,\T} :=(t_{k+d+1}-t_k ) \left[  t_k, \ldots ,t_{k+d+1} \right] (.-x)^d_{+},
\]
where $\left[  y_0, \ldots ,y_d \right] f$ stands for the divided difference of the function $f$ at the points $ y_0, \ldots ,y_d$, and $y_+ =\max (0,y)$.

The B-splines $B_{k,d,\T}$, $k=1,\ldots, N$, possess interesting properties. In the following result, we summarize some of them \cite{B01}.
\begin{proposition}
The basis functions $B_{k,d,\T}$, $k=1,\ldots , N$ meet the following properties.
\begin{itemize}
\item[1)] Non-negativity: $B_{k,d,\T} (x) \geq 0$,  $\forall x\in \mathbb{R}$.
\item[2)] Local support: $B_{k,d,\T}$ is non-negative in $[t_k, t_{k+d+1}]$ and equal to zero outside.
\item[3)] Partition of unity: $\sum_k B_{k,d,\T} =1$.
\item[4)] The B-splines $B_{k,d,\T}$, $k=1,\ldots , N$ are linearly independent on $\mathbb{R}$.
\end{itemize}
\end{proposition}

In addition, the B-splines $B_{k,d,\T}$ can be computed by means of Cox-Boor recurrence formula. In fact, The k-th B-spline $B_{k,0,\T}$ of degree $0$ for the knot sequence $\T$ is the characteristic function of the half-open interval $[t_k, t_{k+1})$, i.e., the function given by 
\[
B_{k,0,\T} (x) := \begin{cases} & 1,\quad t_k \leq x <t_{k+1}, \\ & 0, \quad \mbox{otherwise} . \end{cases}
\]
And,
\[
B_{k,d,\T} (x) := \frac{x-t_k}{t_{k+d}-t_{k}} B_{k,d-1,\T} (x) +  \frac{t_{k+d+1}-x}{t_{k+d+1}-t_{k+1}} B_{k+1,d-1,\T} (x),\quad d>0,
\]
and by convention the value $0$ is assigned when the term $0/0$ appears. 

It holds,
\[
S_{d,\T,\M} = span \left\lbrace B_{k,d,\T} (x),\quad x\in [t_{d}, t_{N+1}]  \right\rbrace 
\]
Every spline $s \in S_{d,\T,\M}$ can be expressed as: 
\[
s (x) = \sum_{k=1}^N c_k B_{k,d,\T} (x),
\] where the real coefficients $c_k$ are called the control coefficients.

\subsection{A discrete spline quasi-interpolants of degrees $d=2, 3$}

In this subsection, we propose two discrete spline quasi-interpolation schemes, $Q^n_2$ and $Q^n_3$, reproducing quadratic and cubic polynomials, respectively. That is,
\begin{equation}\label{QI_Rep}
Q^n_d p = p,\quad \text{ for  all } p \in \mathbb{P}_d, \quad d=2, 3.
\end{equation}
Then, we will write the both schemes in the quasi-Lagrange form. For a sake of simplicity we consider a uniform partition, i.e., $\T =\left\lbrace x_0[d+1], x_1, \ldots , x_i, \ldots, x_{n-1}, x_n [d+1]\right\rbrace$, by $x_j[d+1]$, $j=0, n$, we mean that $x_j$, $j=0, n$, are repeated $d+1$ times, and then we omit $\T$ from the B-spline definition, i.e, $B_{i,d} = B_{i,d,\T}$.

Define,
\begin{equation}\label{QI}
Q^n_d \left( f\right) = \sum_{i=0}^{n+d-1} \lambda_{i,d} \left( f\right) B_{i,d},
\end{equation}
where the functionals $\lambda_{i,d} \left( f\right)$ are selected to ensure that (\ref{QI_Rep}) holds. These functionals vary depending on the nature of the information available about the function $f$ to be approximated. Typically, they take the form of point, derivative, or integral linear forms. In this work, the first case is considered, wherein $\lambda_{i,d} \left( f\right)$ represents a finite linear combination of discrete values of $f$. 

\subsubsection{Quadratic spline quasi-interpolation schemes}
Define,
\[
{\bf \sigma} := \left\lbrace \sigma_0 = x_0, \, \sigma_i = \frac{1}{2} \left( x_i+x_{i-1}\right),\quad 1\leq i \leq n,\, \sigma_{n+1} = x_{n+1} \right\rbrace .
\]
With ${\bf \sigma }$ we can associate the so-called Schoenberg operator \cite{MS70}, i.e, $\lambda_{i,2} \left( f\right) = f(\sigma_i)$. However, this operator only reproduce linear polynomials, and then is only of convergence order equal to $2$. For this, we define the functionals $\lambda_{i,2} \left( f\right)$ as follows \cite{S03},
\begin{align}
\lambda_{0,2} \left( f\right) &= f(\sigma_0),\nonumber\\
\lambda_{i,2} \left( f\right) &= -\frac{1}{8} f(\sigma_{i-1})+ \frac{5}{4} f(\sigma_{i}) - \frac{1}{8} f(\sigma_{i+1}), \quad 1 \leq i \leq n\label{lamda2}\\
\lambda_{n+1,2} \left( f\right) &= f(\sigma_{n+1}),\nonumber .
\end{align}
Let $\tau := x_{i+1}-x_i$ stands for the space step-length. The result below holds.
\begin{theorem}
Let $Q_{2}^n$ be the spline scheme defined by (\ref{QI}) and (\ref{lamda2}), the for any $f \in C^{3}$, we have
\[
\Vert f- Q_{2}^n \left( f\right) \Vert_{\infty} \leq C  \tau^3 \Vert f^{(3)} \Vert_{\infty} ,
\]
where $C$ represents a positive constant independent of $\tau$.
\end{theorem}

The spline scheme $Q_{2}^n$  can be writing in quasi-Lagrange form:
\[
Q_{2}^n \left( f \right) = \sum_{i=0}^{n+1} f\left( \sigma_i\right) \widetilde{B}_{i,2}.
\]
where the basis functions $\widetilde{B}_{i,2}$ are defined as follows. For $3 \leq i \leq n-2$ we have
\[
\widetilde{B}_{i,2}  := - \frac{1}{8} B_{i-1,2} + \frac{5}{4} B_{i,2} - \frac{1}{8} B_{i+1,2},
\]
and for $i=0, 1, 2$ we define,
\begin{align*}
\widetilde{B}_{0,2} & := B_{0,2} - \frac{1}{3} B_{1,2},\\
\widetilde{B}_{1,2} & := \frac{3}{2} B_{1,2} - \frac{1}{8} B_{2,2},\\
\widetilde{B}_{2,2} & := - \frac{1}{6} B_{1,2} + \frac{5}{4} B_{2,2} - \frac{1}{8} B_{3,2}.
\end{align*}
Similarly, we define,
\begin{align*}
\widetilde{B}_{n-1,2} & := - \frac{1}{8} B_{n-2,2} + \frac{5}{4} B_{n-1,2} - \frac{1}{6} B_{n,2},\\
\widetilde{B}_{n,2} & := -\frac{1}{8} B_{n-1,2} +\frac{3}{2} B_{n,2},\\
\widetilde{B}_{n+1,2} & := -\frac{1}{8} B_{n,2} + B_{n+1,2}.
\end{align*}

Next, we will provide a cubic spline quasi-interpolation scheme.

\subsubsection{Cubic spline quasi-interpolation schemes}
Following the same strategy as for quadratic case, we define \cite{BSTT15},
\begin{footnotesize}
\begin{align*}
\widetilde{B}_{0,3} & := B_{0,3} +\frac{7}{18} B_{1,3}-\frac{1}{6} B_{2,3},\\
\widetilde{B}_{1,3} & := B_{1,3} +\frac{4}{3} B_{2,3}-\frac{1}{6} B_{3,3},\\
\widetilde{B}_{2,3} & :=-\frac{1}{2} B_{1,3} -\frac{1}{6} B_{2,3}+\frac{4}{3} B_{3,3}-\frac{1}{6} B_{4,3},\\
\widetilde{B}_{3,3} & :=\frac{1}{9} B_{2,3} -\frac{1}{6} B_{3,3}+\frac{4}{3} B_{4,3}-\frac{1}{6} B_{5,3},\\
\widetilde{B}_{i,3} & :=-\frac{1}{6} B_{i-1,3} +\frac{4}{3} B_{i,3}-\frac{1}{6} B_{i+1,3}, 4\leq i \leq n-2\\
\widetilde{B}_{n-1,3} & :=-\frac{1}{6} B_{n-3,3} +\frac{4}{4} B_{n-2,3}-\frac{1}{6} B_{n-1,3}+\frac{1}{9} B_{n,3},\\
\widetilde{B}_{n,3} & :=-\frac{1}{6} B_{n-2,3} +\frac{4}{3} B_{n-1,3}-\frac{1}{6} B_{n,3}-\frac{1}{2} B_{n+1,3},\\
\widetilde{B}_{n+1,3} & := -\frac{1}{6} B_{n-1,3} +\frac{4}{3} B_{n,3}+ B_{n+1,3},\\
\widetilde{B}_{n+2,3} & := -\frac{1}{6}B_{n,3} +\frac{7}{18} B_{n+1,3} + B_{n+2,3}.
\end{align*}
\end{footnotesize}
The cubic quasi-interpolation schemes $Q_3^n $ is defined in quasi-Lagrange form as follows,
\begin{equation}
Q_3^n \left( f \right) = \sum_{i=0}^{n+2} f(x_i) \widetilde{B}_{i,3}.
\end{equation}

The following result holds.
\begin{theorem}
For any $f \in C^4$, we have
\[
\Vert f - Q_3^n \left( f\right) \Vert_{\infty} \leq C \tau^4 \Vert f^{(4)}\Vert_{\infty}.
\]
where $C$ represents a positive constant independent of $\tau$.
\end{theorem}

\subsection{Numerical solving of the state equation}
Now, a collocation method based on the spline quasi-interpolation schemes described above will be used to numerically solve the differential equation (\ref{problem}) on $[0, 1]\times [0, T]$.

To obtain an approximate solution of the time-dependent problem (\ref{problem}), we approximate the solution $u(x, t)$ and the control function $h(x,t)$ in terms of B-spline as follows:
\begin{equation}
u(x, t) = \sum_{i=0}^{n+d-1} u_i (t) \widetilde{B}_{i,d},\quad h(x, t) = \sum_{i=0}^{n+d-1} h_i (t) \widetilde{B}_{i,d}, \quad d=2, 3,
\end{equation}
where $u_i (t) = u(x_i, t)$, $h_i (t) = h(x_i, t)$ are unknowns functions to be determined. The boundary conditions in (\ref{problem}) yields
\[
u_0(t) =0,\quad u_{n+d-1} (t) = 0,\quad t\in [0, T].
\]
While at the interior points $x_i$, $i=1,\ldots, n-1$, the condition in \eqref{problem} are of the form,
\begin{footnotesize}
\begin{align*}
\frac{\partial u}{\partial t} (x_i) &= \sum_{j=0}^{n+d-1} u_j (t) x_i^\alpha \widetilde{B}^{\prime\prime}_{j,d} (x_i) + \sum_{j=0}^{n+d-1} u_j (t) \alpha x_i^{\alpha -1}  \widetilde{B}^{\prime}_{i,d} (x_i) + \sum_{j=0}^{n+d-1} u_j (t) \frac{\mu}{x_i^{2-\alpha}} \widetilde{B}_{i,d} (x_i) - \sum_{i=\ell_0}^{\ell_1} h_i (t) \widetilde{B}_{i,d} (x_i)\\
 &= \sum_{j=1}^{n+d-2} u_j (t) \left( x_i^\alpha \widetilde{B}^{\prime\prime}_{j,d} (x_i) +\alpha x_i^{\alpha -1}  \widetilde{B}^{\prime}_{i,d} (x_i) +  \frac{\mu}{x_i^{2-\alpha}} \widetilde{B}_{i,d} (x_i) \right) - \sum_{j=\ell_0}^{\ell_1} h_j (t) \widetilde{B}_{j,d} (x_i)\\
& + \sum_{j \in \{0, n+d-1\}} u_j (t) \left( x_i^\alpha \widetilde{B}^{\prime\prime}_{j,d} (x_i) +\alpha x_i^{\alpha -1}  \widetilde{B}^{\prime}_{i,d} (x_i) +  \frac{\mu}{x_i^{2-\alpha}} \widetilde{B}_{i,d} (x_i) \right).
\end{align*}
\end{footnotesize}
Using the boundary conditions, one can get
\begin{equation}\label{collocation1}
\frac{\partial u}{\partial t} (x_i) = \sum_{j=1}^{n+d-2} u_j (t) \left( x_i^\alpha \widetilde{B}^{\prime\prime}_{j,d} (x_i) +\alpha x_i^{\alpha -1}  \widetilde{B}^{\prime}_{i,d} (x_i) +  \frac{\mu}{x_i^{2-\alpha}} \widetilde{B}_{i,d} (x_i) \right) - \sum_{j=\ell_0}^{\ell_1} h_j (t) \widetilde{B}_{j,d} (x_i).
\end{equation}
The indices $\ell_0 \leq i \leq \ell_1$ stand for the indices of the B-splines $\widetilde{B}_{i,d}$ that have non-zero value at $\overline{b}$.

We should determine the $n+d-2$ unknowns $u_i (t)$, $i=1, \ldots , n+d-2$. To this end, (\ref{collocation1}) can reformulated in matrix form as
\begin{equation*}
{\bf U}^\prime (t) = {\bf A } {\bf U}(t) - {\bf B} {\bf  H}(t),
\end{equation*}
where ${\bf A }_{i, j} = \left( x_i^\alpha \widetilde{B}^{\prime\prime}_{j,d} (x_i) +\alpha x_i^{\alpha -1}  \widetilde{B}^{\prime}_{i,d} (x_i) +  \frac{\mu}{x_i^{2-\alpha}} \widetilde{B}_{i,d} (x_i) \right)$, ${\bf U} = \left[u_1 (t), \ldots , u_{n+d-2} \right] $,   ${\bf B}_{i,j} = \widetilde{B}_{j,d} (x_i)$, and ${\bf H} = \left[0, \ldots , h_{\ell_0} (t), \ldots, h_{\ell_1} (t), \ldots, 0 \right] \in \mathbb{R}^{n+d-2} $.

We will apply the control at $\overline{b} = 1/2$, the midpoint of the interval $[0, 1]$. We consider the following sets of knots:

\begin{align*}
&{\bf x}_1 := \frac{1}{10}\left\lbrace 1, 2, 3, 4, 5, 6, 7, 8, 9, 10 \right\rbrace, \\
&{\bf x}_2 := \frac{1}{10}\left\lbrace 1, 2, 3, 4, 5, 5, 6, 7, 8, 9, 10 \right\rbrace, \\
&{\bf x}_3 := \frac{1}{10}\left\lbrace 1, 2, 3, 4, 5, 5, 5, 6, 7, 8, 9, 10 \right\rbrace.
\end{align*}

In these sets, the boundary knots $0$ and $1$ are repeated degree $+ 1$, where "degree" refers to the spline degree. We will examine both quadratic and cubic splines. Figure \ref{Spl_Knots} shows, the quadratic and cubic B-splines defined on the sets ${\bf x}_i$, $i=1, 2, 3$, respectively.

\begin{figure}[!h]
\centering
\includegraphics[scale=0.5]{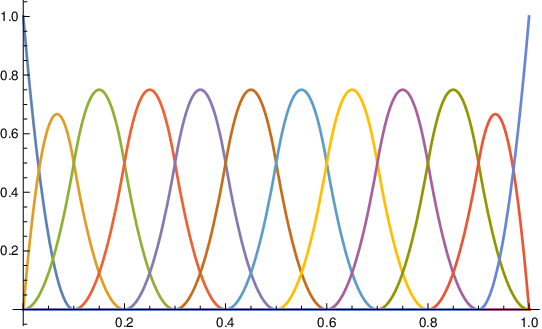} \includegraphics[scale=0.5]{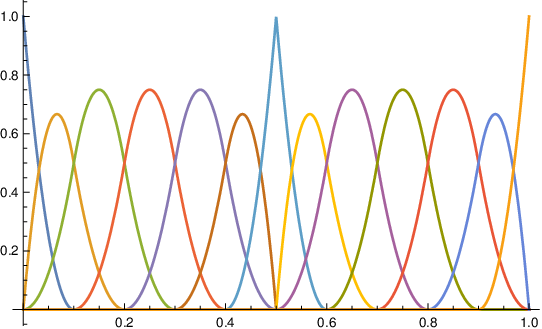} \includegraphics[scale=0.5]{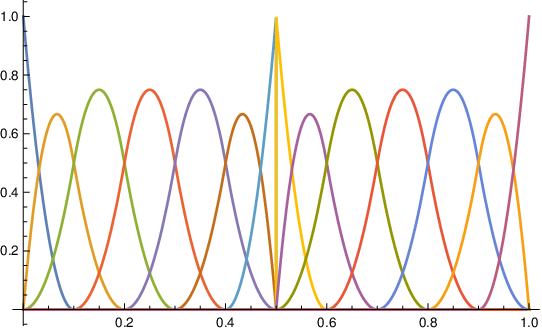}
\includegraphics[scale=0.5]{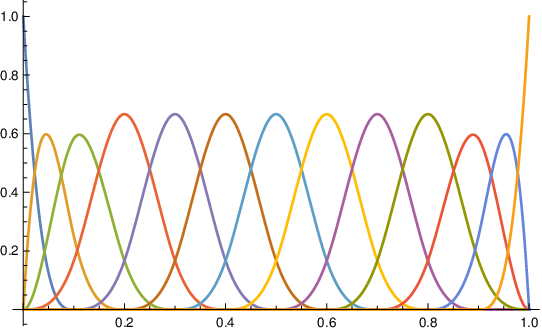} \includegraphics[scale=0.5]{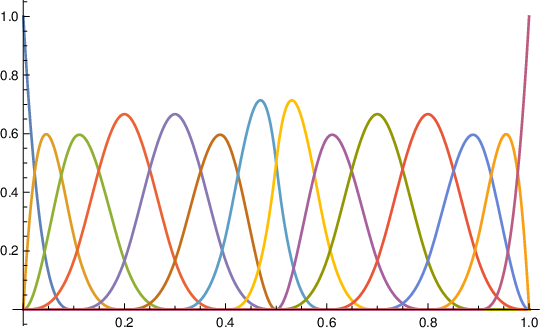} \includegraphics[scale=0.5]{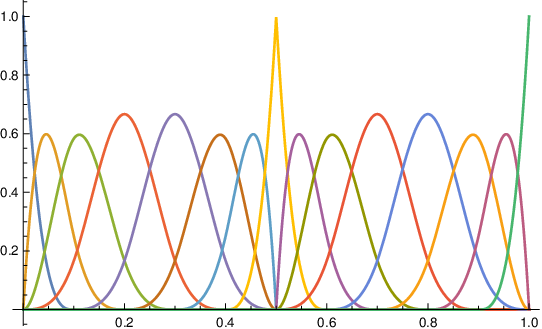}
\caption{Plots of the quadratic (top) and the cubic (bottom) B-splines associated with the sets ${\bf x}_1$ (left), ${\bf x}_2$ (center) and ${\bf x}_3$ (right).}\label{Spl_Knots}
\end{figure}

The finite element matrix ${\bf A}$ is sparse and has good shapes, thanks to the B-spline basis functions having minimum possible supports. Figure \ref{fig2} illustrates schematic representations of the matrix ${\bf A}$ for quadratic B-splines associated with the knot sets ${\bf x}_i$, $i=1, 2, 3$, and for $(\alpha,\mu)=\left( \frac{1}{2}, \frac{1}{16}\right)$.

\begin{figure}[H]
\centering
\includegraphics[scale=0.5]{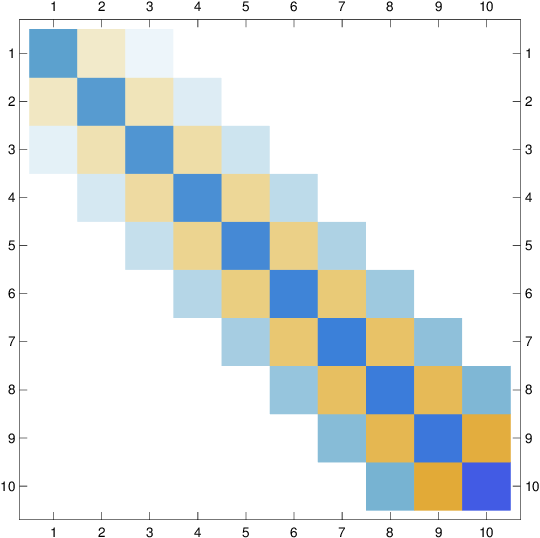} \includegraphics[scale=0.5]{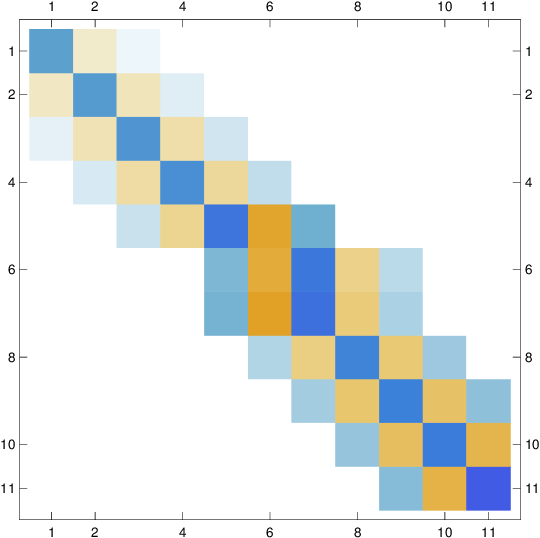} \includegraphics[scale=0.5]{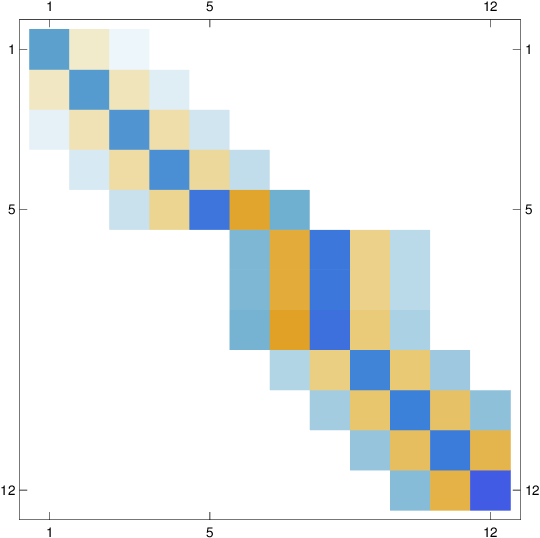}
\caption{Plots of the matrix ${\bf A}$ generated by using quadratic B-splines associated with the knot sets ${\bf x}_i$, $i=1, 2, 3$ (from left to right) and for $\left( \alpha, \mu \right)=\left( \frac{1}{2}, \frac{1}{16}\right)$.}\label{fig2}
\end{figure}

In what follows, we will provide some numerical tests based on the spline quasi-interpolation schemes proposed here and the collocation method described above. In the following examples, we assume that $T=1$, and $\Delta t = \frac{1}{50}$ denotes the time step used for the time discretization.

In the following, we present two examples. In the first example, we utilize the quadratic quasi-interpolating scheme $Q_2^n$, while in the second one, we employ the cubic scheme $Q_3^n$.

\begin{example}
In this example, we consider the initial state $u_0(x) = 3 \sin(2 \pi x)$. We provide numerical tests for $\left( \alpha , \mu\right) = \left( \frac{1}{2}, \frac{1}{16}\right)$.

Figure \ref{fg1} shows the results obtained for the state $u(x,t)$ (left) at $t=1$ and the control $h(x,t)$ (right) at different time instants. 

\begin{figure}[h!]
\centering
\includegraphics[scale=0.6]{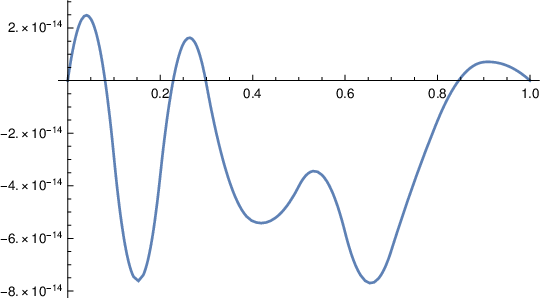}\quad \includegraphics[scale=0.6]{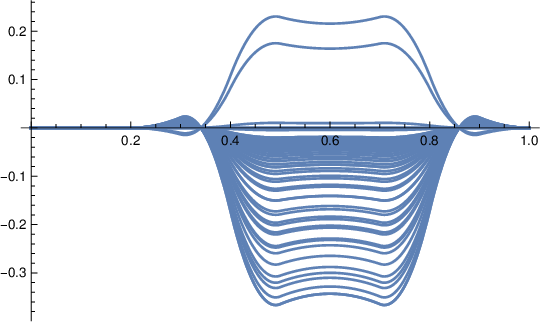}\\
\includegraphics[scale=0.6]{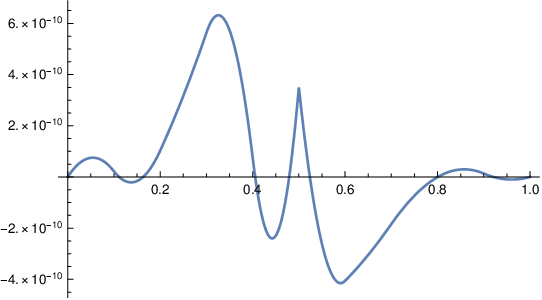}\quad \includegraphics[scale=0.6]{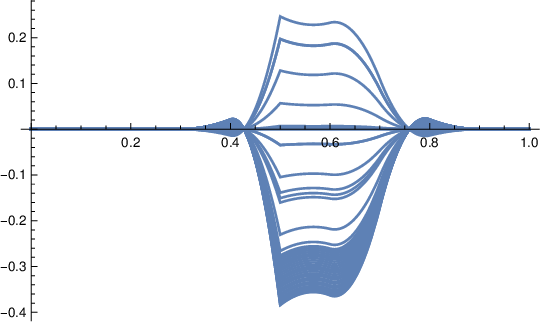}\\
\includegraphics[scale=0.6]{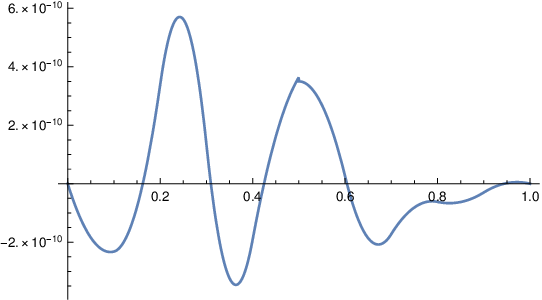}\quad \includegraphics[scale=0.6]{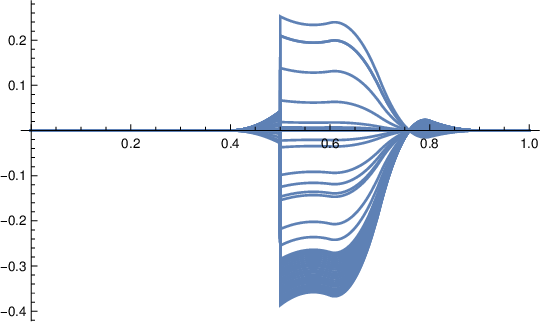}
\caption{Example 1, the state at $t =1$ (left) and the control (right) at different time instants for quadratic splines associated with the knot sets ${\bf x}_i$, $i=1, 2, 3$ from top to bottom, respectively.}\label{fg1}
\end{figure}

\end{example}

\begin{example}
In this example, we take the initial state $u_0 (x) = \sin (3 \pi  x) \cos \left(\frac{1}{2} \pi  (1-x)\right)$. We provide numerical tests for $\left( \alpha , \mu\right) = \left( \frac{1}{8}, \frac{49}{256}\right)$.

Figure \ref{fg2} illustrates the results obtained for the state $u(x,t)$ (left) at $t=1$ and the control $h(x,t)$ (right) at different time instants.

\begin{figure}[h!]
\centering
\includegraphics[scale=0.5]{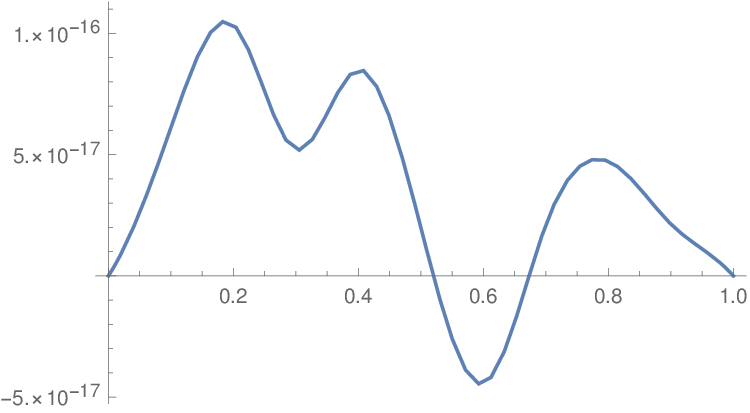} \includegraphics[scale=0.5]{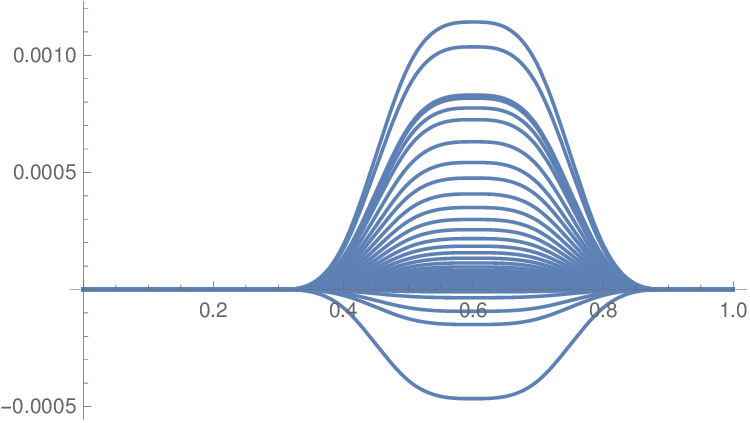}\\
\includegraphics[scale=0.5]{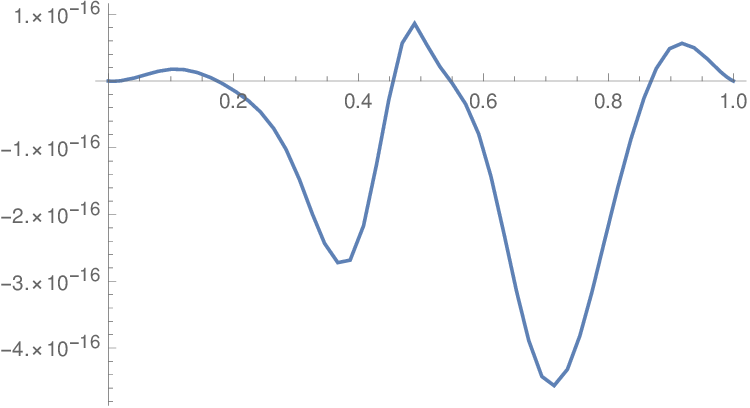} \includegraphics[scale=0.5]{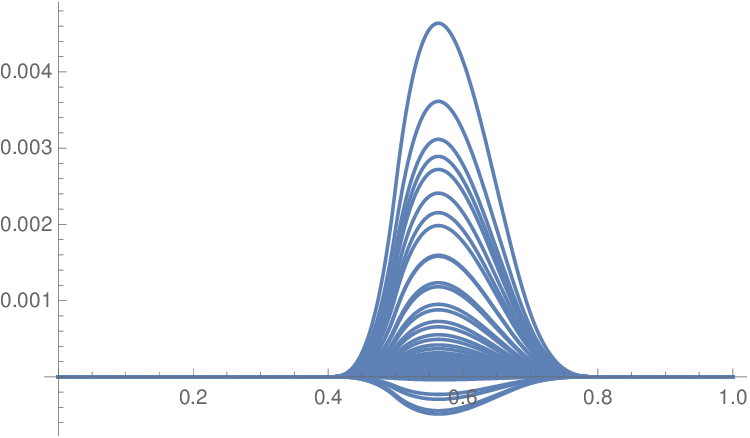}\\
\includegraphics[scale=0.5]{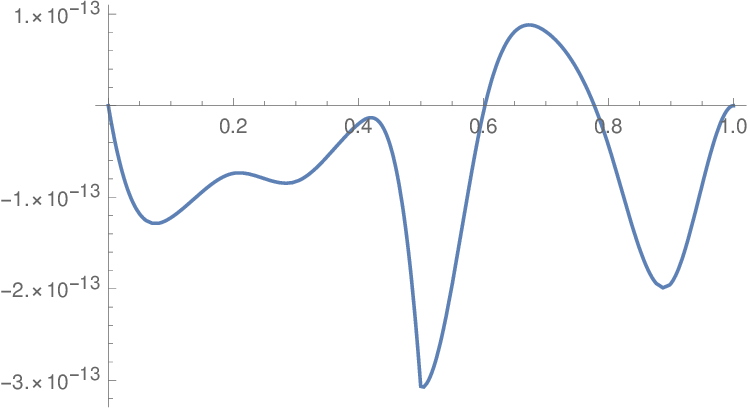}  \includegraphics[scale=0.5]{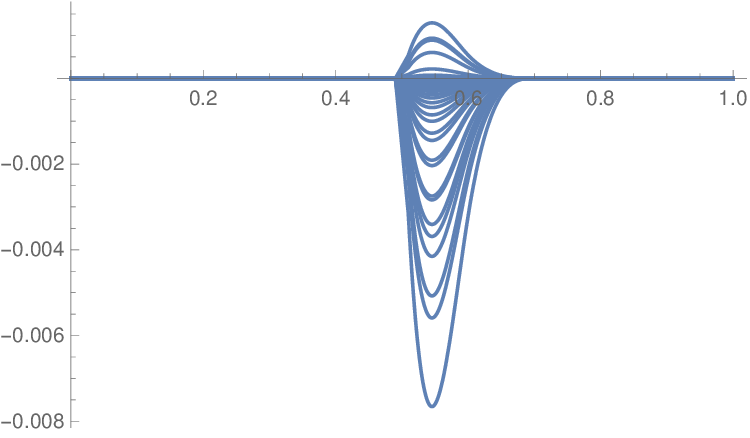}
\caption{Example 2, the state at $t =1$ (left) and the control (right) at different instants for cubic splines associated with the knot sets ${\bf x}_i$, $i=1, 2, 3$ from top to bottom, respectively.}\label{fg2}
\end{figure}

\end{example}

It is evident from the two examples that the support of the B-splines used affects the behaviour of the control and also defines the support of the control. Indeed, we observe that when the point $\bar{b}$ lies within the intersection of more splines, the simulation becomes more efficient. Meaning that for a fixed degree, the B-splines with maximal smoothness are highly suitable for numerical solution of point-wise control problems.

\appendix
\section{Proof of Lemma \ref{gapresult}}\label{appendix:b}
In this appendix we will provide the asymptotic behavior of the sequence of eigenvalues $\left(\lambda_{\alpha, \mu, k}\right)_{k \geq 1}$. To be precise, we will give the sketch of the proof for Lemma \ref{gapresult}.
\begin{proof}
\begin{enumerate}
\item Depending on the case of $\nu(\alpha,\mu)$, the first point of Lemma \ref{gapresult} can be easily deduced from \eqref{boundcase1} and \eqref{boundcase2} together with the description of the eigenvalues in \eqref{eigenvalues}.
\item For the second point, it can be obtained as in \cite{AHSS2022}. Thanks to \cite[Proposition 7.8]{KL2005} and \eqref{boundcase1} or \eqref{boundcase2} (depending on the case of $\nu(\alpha,\mu)$). One has \eqref{gap} holds true with
\begin{align*}
\rho= \begin{cases} \frac{7}{64} \pi^2 (2-\alpha)^2, &\text{ if } 0 \leq \nu(\alpha,\mu) \leq 1/2\\[0.1cm]
(\frac{2-\alpha}{2})^2\pi^2, &\text{ if } \nu(\alpha,\mu) \geq 1/2. \end{cases}
\end{align*}
\end{enumerate}			
\end{proof}

\section{Additional result on the minimal time $T^{(\alpha, \mu)}_0(\overline{b})$}\label{appendix:a}
This appendix will be dedicated to prove that the minimal time $T^{(\alpha, \mu)}_0(\overline{b})$ associated to the null controllability of the equation \eqref{problem} is well-defined and $T^{(\alpha, \mu)}_0(\overline{b}) \in [0,+\infty]$. One has
\begin{theorem}
Let $\mu\leq \mu(\alpha)$ and $y_0 \in L^2(0,1)$. Assume that \eqref{cond1} holds and let $T^{(\alpha, \mu)}_0(\overline{b})$ be the quantity given in \eqref{mintime}. One has $$T^{(\alpha, \mu)}_0(\overline{b}) \in [0,+\infty].$$
\end{theorem}
\begin{proof}
We can deduce from condition \eqref{cond1} that $\Phi_{\alpha,\mu, k}(\overline{b}) \neq 0,\, \text{ for all }k\geq 1.$
Thus,
\begin{align*}
0< |\Phi_{\alpha,\mu, k}(\overline{b})|
\leq \|\delta_{\overline{b}}\|_{H_{\alpha, 0}^{-1, \mu}} \|\Phi_{\alpha,\mu, k}\|_{H_{\alpha, 0}^{1, \mu}} =\|\delta_{\overline{b}}\|_{H_{\alpha, 0}^{-1, \mu}} \sqrt{\lambda_{\alpha,\mu, k}}.
\end{align*}
Then,
\begin{equation*}
\exists \xi>0, \quad 0< |\Phi_{\alpha,\mu, k}(\overline{b})| \leq \xi \lambda_{\alpha,\mu, k},\qquad \forall k\geq 1.
\end{equation*}
and therefore
\begin{equation*}
-\frac{\log(|\Phi_{\alpha,\mu, k}(\overline{b})|)}{\lambda_{\alpha,\mu, k}}
\geq - \frac{\log(\xi \lambda_{\alpha,\mu, k})}{\lambda_{\alpha,\mu, k}},\qquad \forall k\geq 1.
\end{equation*}
Since $\lambda_{\alpha,\mu,k} \rightarrow +\infty\quad \text{as}\quad k\rightarrow +\infty$, then $T^{(\alpha, \mu)}_0(\overline{b}) \in [0,+\infty]$.
\end{proof}


\begin{thebibliography}{99}
\bibitem{AHSS2022}
B. Allal, A. Hajjaj, J. Salhi, and A. Sbai,
{\it Boundary controllability for a coupled system of degenerate/singular parabolic equations}, Evolution Equations and Control Theory, vol. 11, no. 5, p. 1579, 2022.

\bibitem{allal2020}
B. Allal and J. Salhi,
{\it Pointwise Controllability for Degenerate Parabolic Equations by the Moment Method},
Journal of Dynamical and Control Systems, vol. 26, no. 2, pp. 349–362, Feb. 2020.



\bibitem{Khodja2016}
F. Ammar-Khodja, A. Benabdallah, M. Gonz\'alez-Burgos, L. de Teresa,
{\it New phenomena for the null controllability of parabolic systems: Minimal time and geometrical dependence},
J. Math. Anal. Appl. 444 (2016), no. 2, 1071-1113.



\bibitem{ABGT2011}
F. Ammar-Khodja, A. Benabdallah, M. Gonz\'alez-Burgos, L. de Teresa,
{\it The Kalman condition for the boundary controllability of coupled parabolic systems. Bounds on biorthogonal families to complex matrix exponentials},
J. Math. Pures Appl. (9) 96 (2011), 555-590.

\bibitem{biccsantavan2020}
U. Biccari, V. Hernández-Santamaría, J. Vancostenoble,
{\it Existence and cost of boundary controls for a degenerate/singular parabolic equation},
Mathematical Control and Related Fields, 2022, 12(2): 495-530.


\bibitem{coron}
J.-M. Coron, {\it Control and Nonlinearity},
Mathematical Surveys and Monographs, 136, American Mathematical Society, Providence, RI, 2007.


\bibitem{dol}
S. Dolecki,
{\it Observability for the one-dimensional heat equation}, Studia Math. 48 (1973), 291-305.


\bibitem{fatrus1971}
H.O. Fattorini and D.L. Russell,
{\it Exact controllability theorems for linear parabolic equations in one space dimension},
Arch. Ration. Mech. Anal. 43 (1971) 272-292.

\bibitem{FBT}
E. Fern\'andez-Cara, M. Gonz\'alez-Burgos, L. de Teresa,
{\it Boundary controllability of parabolic coupled equations},
J. Funct. Anal. 259 (7) (2010) 1720-1758.



\bibitem{gms2024}
G. Fragnelli, D.Mugnai, A. Sbai
{\it Boundary controllability for degenerate/singular hyperbolic equations in nondivergence form with drift,} arXiv:2402.18247 [math.AP]

\bibitem{gms2024stab}
G. Fragnelli, D. Mugnai, A. Sbai,
{\it Stabilization for degenerate equations with drift and small singular term,} arXiv:2403.17802 [math.AP]




\bibitem{KL2005}
V. Komornik and P. Loreti,
{\it Fourier series in control theory}, Springer, Berlin, 2005.



\bibitem{lions1971}
JL. Lions 
{\it Optimal control of systems governed by partial differential equations}, Berlin: Springer; 1971.


\bibitem{LM2008}
L. Lorch, M.E. Muldoon,
{\it Monotonic sequences related to zeros of Bessel functions}, Numer. Algor 49 (2008), 221-233.





\bibitem{Van2011} J. Vancostenoble,
{\it Improved Hardy-Poincar\'e inequalities and sharp Carleman estimates for degenerate/singular parabolic
problems}, Discrete Contin. Dyn. Syst. Ser. S 4 (2011), 761-790.


\bibitem{VazZua2000} J. L. Vazquez and E. Zuazua,
{\it The Hardy inequality and the asymptotic behaviour of the heat equation with an
inverse-square potential}, J. Funct. Anal. 173, 1 (2000),
103-153.

\bibitem{Watson}
G. N. Watson,
{\it A treatise on the theory of Bessel functions}, second edition, Cambridge University Press, Cambridge, 1944.


\bibitem{B01}C. de Boor, A Practical Guide to Splines (revised edition), Springer, New York (2001)

\bibitem{BLL19}S. Bouhiri, M. Lamnii, A. Lamnii, Cubic quasi-interpolation spline collocation method for solving convection-diffusion equations, Mathematics and Computers in Simulation 164 (2019) 33--45.

\bibitem{MS70} J. M. Marsden, I. J. Schoenberg, An identity for spline functions with applications to variation diminishing spline approximation, J. Approx. Theory 3 (1970), 7--49.

\bibitem{S03}P. Sablonniere, Quadratic spline quasi-interpolants on bounded domains of $\mathbb{R}^d$, $d=1, 2, 3$, Rend. Sem. Mat. Univ. Pol. Torino 61, 3 (2003).

\bibitem{BSTT15}A. Boujraf, D. Sbibih, M. Tahrichi, A. Tijini, A super-convergent cubic spline quasi-interpolant and application, Afr. Mat. 26(2015) 1531--1547.
\end{thebibliography}
\end{document}